\documentclass[10pt]{amsart}
\usepackage[title]{appendix}
\usepackage[utf8]{inputenc}
\usepackage{amsmath,amssymb,amsthm,graphics}
\usepackage{colonequals}
\usepackage{bm}
\usepackage{diagmac2}
\usepackage{enumitem}
\usepackage{textcomp}
\usepackage[alphabetic]{amsrefs}
\usepackage[all,cmtip]{xy}
\usepackage{yfonts}
\usepackage{stmaryrd}
\usepackage{mathrsfs}
\usepackage{caption}
\usepackage{listings}
\usepackage{comment}
\usepackage[colorlinks,anchorcolor=blue,citecolor=blue,linkcolor=blue,urlcolor =blue,bookmarksdepth=2]{hyperref}
\usepackage{tikz-cd}
\usepackage{mathtools}
\usepackage[flushleft]{threeparttable} 
\usepackage{multirow, makecell}
\usepackage{lineno} 
\usepackage{soul} 

\calclayout

\usepackage{setspace}

\newtheorem{Thm}{Theorem}[section]
\newtheorem{Lem}[Thm]{Lemma}
\newtheorem{Cor}{Corollary}[Thm]
\newtheorem{Prop}[Thm]{Proposition}

\theoremstyle{definition}

\newtheorem{eg}[Thm]{Example}

\theoremstyle{remark}
\newtheorem{Rmk}{Remark}[Thm]

\newenvironment{pf}{\begin{proof}}{\end{proof}}

\numberwithin{equation}{section}

\newcommand{\Z}{\mathbb{Z}}
\newcommand{\Q}{\mathbb{Q}}

\newcommand{\OO}{\mathcal{O}}

\newcommand{\p}{\mathfrak{p}}
\newcommand{\bp}{\mathfrak{P}}

\newcommand{\cP}{\mathcal{P}}

\newcommand{\id}{\operatorname{id}}

\newcommand{\Nm}{\operatorname{Nm}}

\newcommand{\Gal}{\mathrm{Gal}}

\newcommand{\RNum}[1]{\uppercase\expandafter{\romannumeral #1\relax}}
\def\lrd{\overwithdelims()}

\DeclareFontFamily{U}{wncy}{}
\DeclareFontShape{U}{wncy}{m}{n}{<->wncyr10}{}
\DeclareSymbolFont{mcy}{U}{wncy}{m}{n}
\DeclareMathSymbol{\Sha}{\mathord}{mcy}{"58} 

\newcommand{\duan}[1]{{\color{blue} \sf  [LD :#1]}}

\newcommand{\kelly}[1]{{\color{orange} \sf [KE :#1]}}
\newcommand{\xw}[1]{{\color{cyan} \sf  [XW: #1]}}

\title[The Hilbert exact sequence \& the principal Chebotarev density theorem]{Nonsplitting of the Hilbert exact sequence and the principal Chebotarev density theorem}
\author{Lian Duan}
\author{Kelly Emmrich}
\author{Ning Ma}
\author{Xiyuan Wang}

\address[Lian Duan]{Department of Mathematics, Colorado State University, Fort Collins, Colorado 80523, USA}
\email{lian.duan@colostate.edu}

\address[Kelly Emmrich]{Department of Mathematics, Colorado State University, Fort Collins, Colorado 80523, USA}
\email{kelly.emmrich@colostate.edu }

\address[Ning Ma]{Department of Mathematics, State University of New York at Buffalo, Buffalo, NY 14260, USA}
\email{nma22@buffalo.edu}

\address[Xiyuan Wang]{Department of Mathematics, The Ohio State University, Columbus, OH 43210, USA}
\email{wang.15476@osu.edu}

\makeatletter
\@namedef{subjclassname@2020}{\textup{}2020 Mathematics Subject Classification}
\makeatother

\begin{document}

\begin{abstract}
Let $K/k$ be a finite Galois extension of number fields, and let $H_K$ be the Hilbert class field of $K$. We find a way to verify the nonsplitting of the short exact sequence 
$$
1\to Cl_K\to \Gal(H_K/k){\to}\Gal(K/k)\to 1
$$
by finite calculation. Our method is based on the study of the principal version of the Chebotarev density theorem, which represents the density of the prime ideals of $k$ that factor into the product of principal prime ideals in $K$. We also find explicit equations to express the principal density in terms of the invariants of $K/k$. In particular, we prove that the group structure of the ideal class group of $K$ can be determined by reading the principal densities. 
\end{abstract}

\maketitle

\section{Introduction}\label{Sect: motivation}
Let $K/k$ be a finite Galois extension of number fields, with Galois group $\Gal(K/k)$. Take $H_K$ to be the Hilbert class field of $K$, that is, the maximal unramified abelian extension of $K$. One can check that $H_K$ is also Galois over $k$; let $\Gal(H_K/k)$ be the associated Galois group. Then there is a natural restriction map $\pi: \Gal(H_K/k)\to \Gal(K/k), \tau \mapsto \tau|_{K}$, whose kernel is isomorphic to the class group $Cl_K$ of $K$. Hence we obtain the \emph{Hilbert exact sequence} (HES):
\begin{equation}\label{Eqn: Hilbert_exact_seq}
1\to Cl_K\to \Gal(H_K/k)\overset{\pi}{\to}\Gal(K/k)\to 1.
\end{equation}
The splitting of the Hilbert exact sequence \eqref{Eqn: Hilbert_exact_seq} has been studied by several people. It was once believed by Herz that this sequence always splits when $k=\Q$ \cite{Herz-Cls-fld-1966}. However, Wyman \cite{Wyman-Hilbert-cls-fld-1973} and Gold \cite{Gold-Hilbert-cls-fld-1977} showed this is in general not true unless $k$ has class number one and $K/k$ is cyclic.  Evidence given by Cornell and Rosen \cite[$\S$~2]{Cornell-Rosen-splitting-Hilbert-cls-fld-1988} implies that the splitting of the sequence will fail if $K/k$ is not cyclic. However, for a concrete Galois extension $K/k$, the splitting of its HES is not easy to check since it is determined by its image in group cohomology $H^2(Cl_K, \Gal(K/k))$. Therefore, one motivation of this paper is to provide an algorithm which could verify the nonsplitting of the HES for certain $K/k$. Such a method is described in our first main result as well as explained in Section \ref{Sect: test_non_split} via a concrete example.

\begin{Thm}[Theorem~\ref{Thm: effective_bound}]\label{Thm: effective_test_nonsplit}
Fix a Galois extension $K/k$. There is a bound $B_{K}$ which only depends on  $n=[K:\Q]$, the  absolute discriminant $|\Delta_K|$ of $K$ and the class number $h_K$ of $K$, such that if any conjugacy class $C$ of $\Gal(K/k)$ cannot be realized by the Frobenius class $\left(\frac{K/k}{\p}\right)$ for at least one unramified prime $\p$ of $k$ satisfying:
	\begin{enumerate}
		\item $N_{K/\Q}(\p)\leq B_{K}$, and 
		\item $\p$ is a product of principal prime ideals in $K$,
	\end{enumerate}
then the Hilbert exact sequence 
$$
1\to Cl_K\to \Gal(H_K/k)\to \Gal(K/k)\to 1
$$
does {not} split. In particular, under the assumption of the Generalized Riemann Hypothesis (GRH), one can take 
$$
 B_K= (4h_K \log |\Delta_{K}|+ n\cdot h_K+5 )^2. 
$$
\end{Thm}

The proof of Theorem~\ref{Thm: effective_test_nonsplit} depends on a refined version of the Chebotarev density theorem. We follow the above notations and take $G=\Gal(K/k)$ and  $E=\Gal(H_K/k)$ for simplification. Denote by $\mathcal{P}_K$ (resp. $\cP_k$) the set of prime ideals of the ring of integers $\OO_K$ (resp. $\OO_k$). For a prime ideal $\mathfrak{P}\in \mathcal{P}_K $, denote by  $N_{K/k}\mathfrak{P}\in \mathcal{P}_k$ the relative norm map (and simply write $N\mathfrak{P}$ if $k=\Q$). For a given conjugacy class $C$ of $G$,  and $\p\in \mathcal{P}_k$ unramified in $K$, let $\cP_{k,C}\coloneqq \left\{\p\in\cP_k | \left(\frac{K/k}{\p}\right)=C\right\}$ be the subset of $\mathcal{P}_k$ which consists of unramfied prime ideals whose associated Frobenius class is equal to $C$.  Recall the definition of the natural density of $\cP_{k,C}$ 
$$
    \mu_{K/k}(C)\coloneqq \lim\limits_{N\to \infty}\frac{\#\left\{\p\in\cP_k |N\p\leq N, \left(\frac{K/k}{\p}\right)=C\right\}}{\#\{\p\in \cP_k|N\p\leq N\}}.
$$
The Chebotarev density theorem states that
\begin{equation}\label{Eqn: classical_Chebotarev}
    \mu_{K/k}(C)=\frac{|C|}{|G|}.
\end{equation}

In our paper, we want to refine $\mu_{K/k}(C)$ into a version which represents the density of primes $\p$ of $k$ which realize the desired conjugacy class $C$ and factor principally in $K$, i.e, we want to consider the following density (if it exists)
\begin{equation}\label{Eqn: principal_density_order_1}
    \mu^1_{K/k}(C)\coloneqq \lim_{N\to \infty}\frac{\#\left\{\p\in\cP_k|N\p\leq N, \left(\frac{K/k}{\p}\right)=C,  \mathfrak{P} \text{ is principal}\right \}}{\#\{\p\in\cP_{k}|N\p\leq N\}},
\end{equation}
where $\mathfrak{P}$ is a prime ideal of $K$ lying above $\p$.  

In Section~\ref{Sect: well_defined}, we will generalize $\mu^1_{K/k}(C)$ as defined in \eqref{Eqn: principal_density_order_1} to $\mu^m_{K/k}(C)$ for every positive integer $m$, and prove they are all well defined. In Section~\ref{Sect: explicit_formula} we describe each of these such densities $\mu^1_{K/k}(C)$ (more generally, $\mu^m_{K/k}(C)$) according to the invariants of $K/k$. To be more precise, for every element $g\in G$, denote by $E_{g}\coloneqq \pi^{-1}(\langle g \rangle)\subset E$, then there is a sub short exact sequence \begin{equation}\label{Eqn: sub_seq_sect_1}
    1\to Cl_K\to E_{g}\to \langle g\rangle\to 1.
\end{equation}
Our next result concerns whether or not $\mu^1_{K/k}(C)$ is positive. 
\begin{Thm}[Theorem \ref{Thm: iff_cond_mu_1,K>0}]\label{Thm: main_thm_positive_density}
Fix a conjugacy class $C$ of $G$, the density $\mu^1_{K/k}(C)>0$ if and only if there exists an element $g\in C$ such that \eqref{Eqn: sub_seq_sect_1} splits.

As a consequence, $\mu^1_{K/k}(C)>0$ for all conjugacy classes if and only if for every maximal cyclic subgroup $U$ of $G$, we have 
$$
    1\to Cl_K\to \pi^{-1}(U)\to U\to 1
$$
splits. In particular, if $K/k$ is cyclic and $k$ has class number one, then $\mu_{K/k}^m(C)>0$ for any $C$ and $m$.  
\end{Thm}
\begin{Rmk}
Theorem~\ref{Thm: iff_cond_mu_1,K>0} is a general version for $\mu^m_{K/k}(C)$. 
\end{Rmk}

Assume that $\mu^1_{K/k}(C)>0$, we also formulate a formulae which explain the densities $\mu_{K/k}^m(C)$ in the style of \eqref{Eqn: classical_Chebotarev}. We summarize it as our third main result. 

\begin{Thm}[Theorem~\ref{Thm: genus_fld_homology_density}]\label{Thm: main_thm_formula}
For a conjugacy class $C$ assume $\mu_{K/k}^1(C)>0$, then 
\begin{equation}\label{Eqn: main_formula_general_C}
    \mu_{K/k}^1(C)=\frac{|C|}{|G|}\frac{|H^1(\langle g\rangle, Cl_K)|}{[K_F:K]}
\end{equation}
where $g$ is an arbitrary element in $C$, $F$ is the fixed field by $g$ and the intermediate field $H_K\supset K_F\supset K$ is the \emph{genus field\footnote{For more details about genus fields, one is referred to \cite{Furuta-genus-fld-1967}.} of $K$ over $F$}. In particular, $\mu_{K/k}^1(\id_G)={|C|}/({|G|}{h_K})$, where $h_K$ is the class number of $K$. 
\end{Thm}

\begin{Rmk}
With our generalized densities $\mu^m_{K/k}(C)$ defined in Section~\ref{Sect: well_defined} we can determine the group structure of the class group $Cl_K$ by reading the associated densities, i.e., taking $C=\{\id_G\}$ to be the trivial class, for every prime power $p^r$ we have
\begin{equation}\label{Eqn: main_formula_id}
    \frac{\mu_{K/k}^{p^r}(\{\id_G\})}{\mu_{K/k}^{p^{r-1}}(\{\id_G\})} =\frac{|Cl_K[p^r]|}{|Cl_K[p^{r-1}]|},
\end{equation}
where for any positive integer $m$, $Cl_K[m]\coloneqq \{x\in Cl_K| x^m=\id\}$. 
\end{Rmk}

\subsection*{Acknowledgements} We would like to thank Jeff Achter, Rachel Pries, Jun Wang and Siman Wong for their helpful insights and suggestions throughout this project.

\section{The robustness of the definition of the principal density}\label{Sect: well_defined}
As in Section~\ref{Sect: motivation}, we have Galois extensions $K/k$, $H_K/K$ and $H_K/k$ with Galois groups $G\coloneqq \Gal(K/k)$, $Cl_K$ and $E\coloneqq\Gal(H_K/k)$, respectively. For any unramified $\p$ of $k$ lying below a prime ideal $\mathfrak{P}$ of $K$, we define the \emph{$K/k$-principal order} of $\p$ to be the smallest positive integer $n_{K/k, \p}$, such that $\bp^{n_{K/k, \p}}$ is principal in $K$. In this case, we say $\p$ is of $K/k$-principal order $n_{K/k, \p}$. Note that this definition is independent of the choice of $\bp$. With this definition, we can consider the following corresponding density (if it exists) for every positive integer $m$, 

\begin{equation}\label{Eqn: principal_density_order_m}
    \mu_{K/k}^{m}(C)\coloneqq \lim_{N\to \infty}\frac{\#\left\{\p\in\cP_k|N\p\leq N, \left(\frac{K/k}{\p}\right)=C,  n_{K/k, \p}| m\right \}}{\#\{\p\in\cP_{k}|N\p\leq N\}}.
\end{equation}
Notice that if we take $m=1$, then this recovers the principal density \eqref{Eqn: principal_density_order_1} in Section~\ref{Sect: motivation}. In Proposition~\ref{Prop: density_well_defn} we will prove that the density $\mu_{K/k}^m(C)$ is well defined for all conjugacy classes $C\subset G$.

\begin{Rmk}\label{Rmk: Theta_density}
For each positive integer $m$, we also can define
\[
\theta_{K/k}^{m}(C)\coloneqq \lim_{N\to \infty}\frac{\#\left\{\p\in\cP_k|N\p\leq N, \left(\frac{K/k}{\p}\right)=C,  n_{K/k, \p}= m\right \}}{\#\{\p\in\cP_{k}|N\p\leq N\}}.
\]
There is a close relation between $\theta_{K/k}^{m}(C)$ and $\mu_{K/k}^{m}(C)$, and almost every result in this section and Section~\ref{Sect: explicit_formula} has a $\theta_{K/k}^{m}(C)$ version.
\end{Rmk}

For a prime ideal $\p\subset \mathcal{O}_k$ which is unramified in $H_K$, assume $\p \OO_K=\prod_{i} \bp_i$. Then we have the next two equations of Frobenius conjugacy classes 
\begin{equation}\label{Eqn: Frobs}
    \pi\left(\frac{H_K/k}{\p}\right)= \left(\frac{H_K/k}{\p}\right)\Big|_K=\left(\frac{K/k}{\p}\right), \quad \left\{ \sigma^{f_{\p}}\Big| \sigma\in \left(\frac{H_K/k}{\p}\right) \right\}=\bigcup_{i}\left(\frac{H_K/K}{\bp_i}\right),
\end{equation}
where $f_\p=[\OO_K/\bp: \OO_k/\p]$ is the inertial degree of $\p$ in $K$. For an element $a$ in a finite group $A$, we denote its order to be $d_A(a)$. By the fact that every element in a conjugacy class $C\subset A$ has the same order, we refer to the common order as $d_A(C)$. 
\begin{Prop}\label{Prop: density_well_defn}
For every conjugacy class $C$ of $G$, and every positive integer $m$, the density $\mu_{K/k}^m(C)$ is well defined and
\begin{equation}\label{Eqn: prin_density_Hilbert_form}
    \mu^m_{K/k}(C)=\frac{|\{\sigma\in E| \pi(\sigma)\in C \text{ and }\sigma^{d_{G}(C)m}=\id_E\}|}{|E|}.
\end{equation}
\end{Prop}

\begin{pf}
Note that $\bp$ is of principal order $n$ if and only if $d_{Cl_K}\left(\frac{H_K/K}{\bp}\right) =n$, or equivalently, if and only if $d_E\left(\frac{H_K/k}{\p}\right) = d_G\left(\frac{K/k}{\p}\right)n$. We have
\[
\begin{split}
    \mu_{K/k}^{m}(C) & = \lim_{N\to \infty}\sum_{n|m}\frac{\#\left\{\p\in\cP_k \big\vert Nm_{k/\Q}(\p)\leq N, \left(\frac{K/k}{\p}\right)=C,  d_E\left(\frac{H/k}{\p}\right)=d_G(C)n \right\}}{\#\{\p\in\cP_{k}|Nm_{k/\Q}(\p)\leq N\}}  \\
    &= \sum_{n|m}\sum_{\substack{\pi(C')=C\\ d_E(C')=d_G(\pi(C'))n }}\lim_{N\to \infty} \frac{\#\left\{\p\in\cP_k \big\vert Nm_{k/\Q}(\p)\leq N, \left(\frac{H_K/k}{\p}\right)=C' \right\}}{\#\{\p\in\cP_{k}|Nm_{k/\Q}(\p)\leq N\}} \\
    &=\sum_{n|m}\sum_{\substack{\pi(C')=C\\ d_E(C')=d_G(\pi(C'))n }} \frac{|C'|}{|E|}\\
    &=\frac{|\{\sigma\in E| \pi(\sigma)\in C \text{ and }\sigma^{d_{G}(C)m}=\id_E\}|}{|E|}.
\end{split}
\]
Where $C'$ is a conjugacy class of $E$.
\end{pf}

By the above proposition and its proof, one sees that the density $\mu^m_{K/k}(C)$ is totally determined by the set (which could be empty)
\begin{equation}\label{Eqn: defn_sets_C_m}
    C_m\coloneqq \{\sigma\in E| \pi(\sigma)\in C \text{ and }\sigma^{d_{G}(C)m}=\id_E\}.
\end{equation}

Before leaving this section, we list some basic properties of the densities which will be used later.

\begin{Lem}\label{Prop: basic_properties_density}
Let $h_K$ be the class number of $K$ and $C$ be a conjugacy class of $G$.
\begin{enumerate}
    \item We have $\mu_{K/k}^{h_K}(C)=\mu_{K/k}(C)$.
    \item If $m_1\mid m_2$, then $\mu_{K/k}^{m_1}(C)\leq \mu_{K/k}^{m_2}(C)$.
    \item For every $m>0$, let $m_{0}=\gcd(m, h_K)$. Then $\mu^m_{K/k}(C)=\mu^{m_0}_{K/k}(C)\leq \mu^{h_K}_{K/k}(C)$.  
\end{enumerate}
\end{Lem}
\begin{pf}
These are easy consequences of the definition of $\mu_{K/k}^{m}(C)$ (See equation \eqref{Eqn: principal_density_order_m}).
\end{pf}

By this lemma, we are only interested in $\mu_{K/k}^{m}(C)$, where $m$ is a divisor of $h_{K}$.

\section{Explicit formula for the densities}\label{Sect: explicit_formula}

Although Proposition~\ref{Prop: density_well_defn} and its proof imply the densities only depend on the the conjugacy classes of $E$, an explicit formula in the style of \eqref{Eqn: classical_Chebotarev} is preferred. This section is devoted to find a concrete formula. First, in subsection~\ref{Sect: when_positive?}, Theorem~\ref{Thm: iff_cond_mu_1,K>0} shows that whether the density $\mu_{K/k}^m(C)$ is nonzero depends on whether some corresponding short exact sequence \eqref{Eqn: sub_seq} splits. Hence when $k$ has class number one, we know $\mu_{K/k}^m(C)>0$ for any $C$ and $m$ by Corollary~\ref{Cor: sufficient_cond_mu_1,K>0} and Proposition~\ref{Cor: Gold_splitting_Hilbert}. We formulate an explanation of $\mu_{K/k}^m(C)$ in terms of $|G|$, $|C|$, the first Tate cohomology and genus degree in Theorem~\ref{Thm: genus_fld_homology_density} and Corollary~\ref{Cor: density_formula_general_C}. In particular, when $C=\{\id_{G}\}$ is the trivial class in $G$, Corollary~\ref{Thm: class_group_density} says one can read the structure of the class group $Cl_K$ by $\mu_{K/k}^m(\{\id_G\})$ as $m$ varies. These formulas and their proofs are given in subsection~\ref{Sect: the_formula}.  

\subsection{When is $\mu_{K/k}^m(C)$ positive?}\label{Sect: when_positive?}
By the discussion at the end of Section \ref{Sect: well_defined}, we are only interested in $\mu_{K/k}^{m}(C)$ for $m$ a divisor of $h_K$. For any element $g\in G$, we can construct the associated subexact sequence (\ref{Eqn: sub_seq_sect_1}) of HES, 
\[
    1\to Cl_K\to E_{g}\to \langle g\rangle\to 1.
\]
Let $Cl^{0}_{K}[n]$ be the subgroup of $Cl_{K}$ generated by the elements of order exactly $n$. This group $Cl_{K}^{0}[n]$ may not exist. If it exists, there is a short exact sequence
\begin{equation}\label{Eqn: sub_seq}
    1\to Cl_K/Cl_{K}^{0}[n] \to E_{g}/Cl_{K}^{0}[n]\to \langle g\rangle\to 1.
\end{equation}

The following theorem concerns the case when $\mu_{K/k}^{m}(C)$ is positive.

\begin{Thm}\label{Thm: iff_cond_mu_1,K>0}
Fix a conjugacy class $C$ of $G$, the density $\mu^m_{K/k}(C)>0$ if and only if there exists a positive divisor $i$ of $m$ such that the short exact sequence 
\begin{equation}
    1\to Cl_K/Cl^{0}_{K}[i]\to E_{g}/Cl_{K}^{0}[i]\to \langle g\rangle\to 1
\end{equation}
exists and splits for some $g\in C$.

As a consequence, $\mu^m_{K/k}(C)>0$ for all conjugacy classes if and only if for every maximal cyclic subgroup $U$ of $G$, there exists a divisor  $i_{U}$ of $m$ such that the short exact sequence
\begin{equation}\label{Eqn: prop_iff_all_C}
    1\to Cl_K/Cl_{K}^{0}[i_U]\to \pi^{-1}(U)/Cl_{K}^{0}[i_U]\to U\to 1
\end{equation}
exists and splits.

\end{Thm}
\begin{pf}

Assume that $\mu^m_{K/k}(C)>0$. Then $C_m$ in \eqref{Eqn: defn_sets_C_m} is nonempty. There is an element $\sigma$ such that $\pi(\sigma)\in C$ and $d_{E}(\sigma)=d_{G}(C)i$ for a positive divisor $i$ of $m$. So the group $Cl_{K}^{0}[i]$ exists and $\pi(\sigma)\mapsto \sigma Cl_{K}^{0}[i]$ defines a splitting of the short exact sequence 
\[
1\to Cl_K/Cl^{0}_{K}[n]\to E_{g}/Cl^{0}_{K}[n]\to \langle \pi(\sigma)\rangle\to 1.\]

For the other direction, assume that there is a split short exact sequence 
\[
1\to Cl_K/Cl^{0}_{K}[i]\to E_{g}/Cl^{0}_{K}[i]\to \langle g\rangle\to 1
\]
for some $g\in C$ and some divisor $i$ of $m$. Let $\sigma Cl_{K}^{0}[i]$ be the image of $g$ under the splitting map. Then $\sigma Cl_{K}^{0}[i]$ and $g$ have the same order. We have $\frac{d_{E}(\sigma)}{d_{G}(C)}\mid m$. So $\sigma$ is an element in $C_m$ and $\mu^{m}_{K/k}(C)>0$.

To see the second part, suppose that $\mu^m_{K/k}(C)>0$ for every $C$. Let $U$ be a maximal cyclic group with a generator $g$. Let $C_{0}$ be the conjugacy class containing $g$. Since $\mu_{K/k}^{m}(C_0)>0$, using the above arguments one can find an element $g_0\in C_0$ such that the associated short exact sequence $ 1\to Cl_K/Cl^{0}_{K}[i]\to E_{g_0}/Cl_{K}^{0}[i]\to \langle g_0\rangle\to 1$ splits. Since $g_0$ and $g$ are in the same conjugacy class $C_0$, there exists an element $h\in G$ such that $g=h^{-1}g_0 h$. Take $\tau\in \pi^{-1}(h)$ to be an arbitrary lift of $h$ in $E$, it is not hard to check that \eqref{Eqn: prop_iff_all_C} can be recovered by conjugating $ 1\to Cl_K/Cl^{0}_{K}[i]\to E_{g_0}/Cl_{K}^{0}[i]\to \langle g_0\rangle\to 1$ by $\tau$. This proves the splitting of \eqref{Eqn: prop_iff_all_C}. Conversely, assume that \eqref{Eqn: prop_iff_all_C} splits for all maximal cyclic subgroups $U$ of $G$. One can deduce the splitting of \eqref{Eqn: prop_iff_all_C} for every (not necessarily maximal) cyclic group $U$. Then for each conjugacy class $C$, take $U$ to be any cyclic subgroup of $G$ which intersects $C$ nontrivially,  and take $g\in U\cap C$. Then the splitting of $1\to Cl_K/Cl^{0}_{K}[i_U]\to E_{g}/Cl^{0}_{K}[i_U]\to \langle g\rangle\to 1$ is deduced from the splitting of \eqref{Eqn: prop_iff_all_C}. So $C_m\neq \emptyset$, hence the proof is done.
\end{pf}

\begin{Cor}\label{Cor: sufficient_cond_mu_1,K>0}
If the Hilbert exact sequence 
$$
1\to Cl_K\to E\overset{\pi}{\to}G\to 1
$$
splits, then $\mu^1_{K/k}(C)>0$ for every conjugacy class $C$. 
\end{Cor}
\begin{pf}
Note that if the Hilbert exact sequence splits, then for every subgroup $H\subset G$, the corresponding subsequence 
$$
1\to Cl_K\to \pi^{-1}(H)\to H\to 1
$$
also splits. In particular, let  $H$ run over all maximal cyclic subgroups $U$ of $G$, then this corollary follows as an immediate consequence of  Corollary~\ref{Thm: iff_cond_mu_1,K>0}. 
\end{pf}

It is worth mentioning that Wyman \cite{Wyman-Hilbert-cls-fld-1973}, Gold \cite{Gold-Hilbert-cls-fld-1977}, Cornell and Rosen \cite{Cornell-Rosen-splitting-Hilbert-cls-fld-1988} proved that the assumption of Corollary~\ref{Cor: sufficient_cond_mu_1,K>0} is true when $k$ has class number one and $K/k$ is cyclic. More precisely, then proved the following result. 

\begin{Prop}\cite[Theorem~1 and 2]{Gold-Hilbert-cls-fld-1977}\label{Cor: Gold_splitting_Hilbert}
The short exact sequence 
$$
1\to \Gal(H_K/K)\to \Gal(H_K/k)\to \Gal(K/k)\to 1
$$
splits if either of the following two conditions is fulfilled. 
\begin{enumerate}
    \item Let $r$ be the least common multiple of the ramification indices of all primes in $K/k$, then $r=[K:k]$.
    \item If $K/k$ is cyclic and $H_k\cap K=k$.
\end{enumerate}
In particular, when there is a totally ramified prime in $K/k$; or when $K/k$ is cyclic and $k$ has class number one, the above sequence splits. 
\end{Prop}

\begin{Rmk}
To the best knowledge of the authors, a sufficient and necessary condition for \eqref{Eqn: prop_iff_all_C} to split for all $U$ is still unknown. Therefore, it is of interest to determine if there is an example of a Hilbert exact sequence which is nonsplit, but where every subsequence of the form \eqref{Eqn: prop_iff_all_C} is split. 
\end{Rmk}

\subsection{Formula for the densities.}\label{Sect: the_formula}
In this subsection, we will try to deduce a formula to compute the density $\mu_{K/k}^m(C)$ for a fixed conjugacy class $C$ of $G$. Until the end of this section, we will always assume
$
    \mu_{K/k}^m(C)>0.
$

Fix an element $\sigma\in E$ such that $\pi(\sigma)\in C$ and $\sigma$ has order dividing $d_{G}(C)m$. Hence $\sigma$ is an element in the set $C_m$ defined at \eqref{Eqn: defn_sets_C_m}. Consider the following group homomorphism (one can check it is well defined), 
\begin{equation}
\label{Eqn: the_norm_map}
    \begin{aligned}
    {\rm N}_{\sigma, m}: Cl_K &\to Cl_K\\
    x & \mapsto (x\sigma)^{d_{G}(C)m}\\
     & =x\sigma \cdot x\sigma \cdots x\sigma \quad (d_{G}(C)m\text{ copies})\\
     &= x \cdot \sigma x \sigma^{-1} \cdot \sigma^2 x \sigma^{-2}\cdots\sigma^{d_{G}(C)m-1} x \sigma^{-({d_{G}(C)m-1)}}.
\end{aligned}
\end{equation}
Then $\ker({\rm N}_{\sigma, m})=\{x\in Cl_K| (x\sigma)^{d_{G}(C)m}=\id_E\}$ and 
$\ker({\rm N}_{\sigma, m})\cdot \sigma=\{\tau\in E| \tau \in C_m \text{ and }\pi(\tau)=\pi(\sigma)\}.$

\begin{Prop}\label{Prop: density_kernel_relation}
Assume that $\mu_{K/k}^{m}(C)>0$.
For any element $\sigma\in \pi^{-1}(C)$, we have
$$
    \mu_{K/k}^m(C)=\frac{|C|}{|G|}\frac{|\ker({\rm N}_{\sigma, m})|}{h_K}.
$$
\end{Prop}
\begin{proof}
Note that $\ker({\rm N}_{\sigma, m})=\tau \ker({\rm N}_{\tau\sigma\tau^{-1}, m})\tau^{-1}$ for any element $\tau\in E$. We have $|\ker({\rm N}_{\sigma, m})|=|\ker({\rm N}_{\tau \sigma\tau^{-1}, m})|$. Moreover, if $\pi(\sigma)\neq \pi(\tau \sigma \tau^{-1})$ then $ \ker({\rm N}_{\sigma, m}) \sigma $ and $ \ker({\rm N}_{\tau \sigma \tau^{-1}, m})(\tau \sigma \tau^{-1})$ are disjoint. By this observation and \eqref{Eqn: defn_sets_C_m}, if for every $g\in C\subset G$, we choose a lift $\sigma_{g}\in \pi^{-1}(g)$ such that $\sigma_g$ is conjugate with the given $\sigma$ in $E$, then 
$$
C_m=\bigsqcup_{g\in C} \ker({\rm N}_{\sigma_g, m})\sigma_g.
$$
So $|C_m|=|C||\ker({\rm N}_{\sigma, m})|$ for $\sigma \in \pi^{-1}(C)$. Combining this with \eqref{Eqn: prin_density_Hilbert_form} proves the result.
\end{proof}

By this proposition, the study of the density $\mu_{K/k}^m(C)$ is reduced to studying the subgroup $\ker({\rm N}_{\sigma, m})$. First we look at the special case $m=1$. Let $g=\pi(\sigma)\in G$, then $\langle g\rangle$ acts on $K$. We denote by $F\coloneqq K^{\langle g\rangle}$ the fixed field of $K$ by $\langle g\rangle$, then by genus theory \cite{Furuta-genus-fld-1967}, there exists an intermediate field $H_K\supset K_F\supset K$ which is maximal among all such possible intermediate fields whose Galois group over $F$ is abelian. This $K_F$ is called the \emph{genus field of $K$ over $F$} . We call the degree $[K_F: K]$ the \emph{genus number} of $K$ over $F$. Now we are equipped to state the following result. 

\begin{Thm}\label{Thm: genus_fld_homology_density}
For every conjugacy class $C$ of $G$ such that $\mu_{K/k}^1(C)>0$,  let $\sigma\in \pi^{-1}( g)\subset E_g$ be an element that has order $d_{G}(C)$ for some $g\in C$,  take $F$ to be the subfield of $K$ fixed by $g$ and take $K_F$ to be the genus field of $K$ over $F$, then 
$$
\frac{|\ker({\rm N}_{\sigma, 1})|}{h_K}=\frac{|{H}^1(\langle g \rangle, Cl_K)|}{[K_F:K]}.
$$
As a consequence, we have 
$$
\mu_{K/k}^1(C)=\frac{|C|}{|G|}\frac{|H^1(\langle g \rangle, Cl_K)|}{[K_F:K]}.
$$
\end{Thm}
\begin{pf}
First the element $\sigma$ exists by Theorem~\ref{Thm: iff_cond_mu_1,K>0}. Note that $E_{g}$ acts on $Cl_K$ naturally by conjugation, and this action factors through $\langle \sigma \rangle\simeq \langle g\rangle $. By the theory of the Tate cohomology/homology \cite[Chap.~\RNum{8}, $\S$1, $\S$4]{Serre-LF}, it follows that
\begin{equation}\label{Eqn: Tate_homology_H0}
    H^1(\langle g \rangle, Cl_K)\simeq H^1(\langle \sigma \rangle, Cl_K)=\hat{H}^1(\langle \sigma \rangle, Cl_K)=\hat{H}_0(\langle \sigma \rangle, Cl_K)=\frac{\ker({\rm N}_{\sigma, 1})}{D(Cl_K)},
\end{equation}
where $D(Cl_K)$ is the subgroup of $\ker({\rm N}_{\sigma, 1})$ generated by $xsx^{-1}s^{-1}$ as  $x$ and $s$ run over $Cl_K$ and $\langle \sigma \rangle$ respectively. Easy calculation shows that $D(Cl_K)$ is the commutator subgroup $[E_g, E_g]$ of $E_g$. Hence using Galois theory, one sees that $D(Cl_K)$ corresponds to the genus field $K_F$. Hence $$[K_F:K]\cdot d_{G}(C)=|\Gal(K_F/F)|=|E_{g}|/|D(Cl_K)|.$$
Applying \eqref{Eqn: Tate_homology_H0}, we get
$$
    |H^1(\langle g \rangle, Cl_K)|= \frac{|\ker({\rm N}_{g, 1})|}{|D(Cl_K)|}=\frac{|\ker({\rm N}_{\sigma, 1})|[K_F:K]d_{G}(C)}{|E_{g}|}.
$$
Since $|E_{g}|=h_K d_{G}(C)$, we thus get 
$$
|H^1(\langle g \rangle, Cl_K)|=\frac{|\ker({\rm N}_{\sigma, 1})|[K_F:K]}{h_K}=[K_F:K]\frac{|\ker({\rm N}_{\sigma, 1})|}{h_K}.
$$
Then the last equation follows immediately from Proposition~\ref{Prop: density_kernel_relation}.
\end{pf}

Now we consider the case $m>1$ (still assuming $\mu^1_{K/k}(C)>0$). By the definition of ${\rm N}_{\sigma, m}$ \eqref{Eqn: the_norm_map} and the fact there is an element $\sigma$ such that $\sigma^{d_{G}(C)}=\id_E$, we know ${\rm N}_{\sigma, m}=({\rm N}_{\sigma, 1})^m: x\mapsto ({\rm N}_{\sigma, 1}(x))^m$. Hence one sees that 
$$
\ker({\rm N}_{\sigma, m})/\ker({\rm N}_{\sigma, 1})=(Cl_K/\ker({\rm N}_{\sigma, 1}))[m]
$$
where for a finite abelian group $\mathcal{G}$,  $\mathcal{G}[m]$ denotes the subgroup of $m$-torsion elements of $\mathcal{G}$. Combining this with Proposition~\ref{Prop: density_kernel_relation} and Theorem~\ref{Thm: genus_fld_homology_density}, we deduce the following corollary. 

\begin{Cor}\label{Cor: density_formula_general_C}
With all the notations above, if $\mu_{K/k}^{1}(C)>0$, we have 
$$
\mu_{K/k}^m(C)=\frac{|C|}{|G|}\frac{|H^1(\langle \sigma \rangle, Cl_K)|}{[K_F:K]} |(Cl_K/\ker({\rm N}_{\sigma, 1}))[m]|.
$$
\end{Cor}

The case $C=\{\id_G\}$ is of special interest, so we state it as a theorem. 
\begin{Cor}\label{Thm: class_group_density}
Taking $C=\{\id_G\}$ to be the trivial conjugacy class in $G$, for every positive integer $m$, we have 
$$
Cl_K[m]=\ker({\rm N}_{\id_E, m}).
$$
In particular, for every prime integer $p$ and every positive integer $r$, we have 
$$
    \frac{\mu_{K/k}^{p^r}(\{\id_G\})}{\mu_{K/k}^{p^{r-1}}(\{\id_G\})} =\frac{|Cl_K[p^r]|}{|Cl_K[p^{r-1}]|}.
$$
In addition, $\mu_{K/k}^1(\id_G)=\frac{|C|}{|G|}\frac{1}{h_K}$.
\end{Cor}
\begin{pf}
This theorem can be deduced from the results by using the facts that $|H^1(\id_E, Cl_K)|=1$, $K_F=H_K$ and $\ker({\rm N}_{\id_E, 1})=\id_E$ in this case. Here we give a more explicit proof. When $C=\{\id_G\}$ is the trivial class, we can always take $\sigma=\id_E$, which has order (in $E$) one. Hence the we always have $\mu^m_{K/k}(\{\id_G\})>0$ for every positive integer $m$. Then it is not hard to see that ${\rm N}_{\id_E, m}(x)=x^m$. Hence $\ker({\rm N}_{\id_E, m})=\{x\in Cl_K| x^m=\id_E\}=Cl_K[m]$. In particular, $\frac{\mu_{K/k}^{p^r}(\{\id_G\})}{\mu_{K/k}^{p^{r-1}}(\{\id_G\})} =\frac{|Cl_K[p^r]|}{|Cl_K[p^{r-1}]|}$ follows by taking $m=p^{r-1}$ and $p^r$ respectively. The last equation follows from Proposition~\ref{Prop: density_kernel_relation}.
\end{pf}

\section{Effective method}\label{Sect: effective_version}
This section is devoted to prove Theorem~\ref{Thm: effective_test_nonsplit}. As we have seen in the previous sections, especially from Corollary~\ref{Cor: sufficient_cond_mu_1,K>0}, if the Hilbert exact sequence \eqref{Eqn: Hilbert_exact_seq} for one Galois extension $K/k$ splits, then $\mu^1_{K/k}(C)>0$ for every conjugacy class $C\subset G$. As one can see from the proof of Proposition~\ref{Prop: density_well_defn}, $\mu^1_{K/k}(C)$ is essentially dependent on the (union) of conjugacy classes of $E=\Gal(H_K/k)$. Thus if one can find a bound $B_K$ such that every conjugacy class of $E$ can be realized as the Frobenius conjugacy class of at least one prime ideal $\p$ of $k$ which is unramified in $H_K$, then every conjugacy class $C$ of $G=\Gal(K/k)$ will have a contribution from at least one such prime ideal. Hence our task is reduced to finding a bound to realize every conjugacy class of $E$. This effective version of the Chebotarev density theorem has been studied by many people, here we use the version of Bach and Sorenson.

Let $L/k$ be a Galois extension of number fields, with $L\neq\Q$. Let $\Delta$ denote the absolute value of $L$'s discriminant (assuming $\Delta\to \infty$). Let $n_L$ denote the degree of $L$ over $\Q$. Let $C\subset \Gal(L/k)$ be a conjugacy class of the Galois group of $L/k$. Assume Generalized Riemann Hypotheses (GRH). Then \cite[Theorem~3.1, 3.2, 5.1 and Corollary~3.3]{BS-Chebotarev} there is an unramified prime ideal $\mathfrak{p}$ of $k$ with $\left(\frac{L/k}{\mathfrak{p}}\right)=C$, of residue degree $1$, satisfying 
$$
N\mathfrak{p}\leq (4\log \Delta + 2.5n_L+5)^2.\footnote{A better bound for each concrete extension can be computed effectively. }
$$

With the above result of Bach and Sorenson, we are left finding/estimating the degree $[H_K:\Q]$ and discriminant $\Delta_{H_K}$. In the following, we will always assume that the extension $K/\Q$ is known in the sense that we know its absolute degree $n$ and absolute discriminant $\Delta_{K}$.

\begin{Thm}\label{Thm: effective_bound}
Let $K/k$ be a Galois extension. Assume $K$ has absolute degree $n$, absolute discriminant $\Delta_K$ and class number $h_K$. Assuming GRH, take 
\begin{equation}\label{bound p}
    B_K= (4h_K \log |\Delta_{K}|+ n\cdot h_K+5 )^2. 
\end{equation}
Then a conjugacy class $C\subset \Gal(K/k)$ satisfies $\mu^1_{K/k}(C)>0$ if and only if there exists an unramified prime ideal $\p$ of $k$ such that 
\begin{enumerate}
    \item $N\p \leq B_K$,
    \item $\p$ factors principally in $K$, 
    \item $\left(\frac{K/k}{\p}\right)=C$. 
\end{enumerate}
In particular, if the associated Hilbert exact sequence splits, then every conjugacy class $C$ can be realized as the Frobenius class by at least one prime ideal $\p$ as above. 
\end{Thm}
\begin{Rmk}
Without the GRH assumption, there is still an effective bound $B_K$ to guarantee the result of the above theorem.  
\end{Rmk}

\begin{proof}
As explained in the beginning of this section, it is sufficient to find a bound to realize every conjugacy class of $\Gal(H_K/k)$. Thanks to the theorem of  Bach and Sorenson, we are left only to bound the absolute degree and the absolute discriminant of $H_K$. In fact, $[H_K:\Q]=h_K\cdot [K:\Q]$. Hence the only nontrivial part of this proof is to give an estimation of $|\Delta_{H_K/\Q}|$. 

For this purpose, we recall \cite[Chapter~\RNum{3}, $\S$~2]{Neukirch-Alg-NT-1999}  that there is a fractional ideal $\mathfrak{D}_{H_K/K}$ of $\OO_{H_K}$ such that a prime ideal $\widetilde{\mathfrak{P}}$ of $\OO_{H_K}$ is ramified over $K$ if and only if $\widetilde{\mathfrak{P}}|\mathfrak{D}_{H_K/K}$ \cite[Chapter~\RNum{3}, Theorem~2.6]{Neukirch-Alg-NT-1999}. Moreover, we have that \cite[Chapter~\RNum{3}, Theorem~2.9]{Neukirch-Alg-NT-1999} the discriminant $\Delta_{H_K/\Q}$ is the norm of the different $\mathfrak{D}_{H_K/\Q}$, i.e., 
$$
\Delta_{H_K/\Q}=N\mathfrak{D}_{H_K/\Q}
$$
and \cite[Chapter~\RNum{3}, Proposition~2.2]{Neukirch-Alg-NT-1999}
$$\mathfrak{D}_{H_K/\Q}=\mathfrak{D}_{H_K/K}\mathfrak{D}_{K/\Q}.$$ 
Note that since $H_K$ is unramified over $K$, we know that $\mathfrak{D}_{H_K/K}=(1)$ is the trivial ideal of $\OO_{H_K}$. Thus $\mathfrak{D}_{H_K/\Q}=\mathfrak{D}_{K/\Q}$ (considered as an ideal of $\OO_{H_K}$). Hence we know that in this case, 
$$
\Delta_{H_K}=N\mathfrak{D}_{H_K/\Q}=N\mathfrak{D}_{K/\Q}\OO_{H_K}=(N\mathfrak{D}_{K/\Q})^{h_K}=\Delta_K^{h_K}.
$$
Applying the theorem of Bach and Sorenson with $L=H_K$, we get an effective bound to cover all the conjugacy classes of $\Gal(H_K/k)$. That is 
\begin{equation*}
    B_K= (4h_K \log |\Delta_{K}|+ n\cdot h_K+5 )^2. 
\end{equation*}
Now we take 
$$S\coloneqq \{\p \mid N\p\leq B_K \text{ and }\p \text{ factors principally in }K\}.$$ 
Then by the same arguments as in the proof of Proposition~\ref{Prop: density_well_defn}, one can easily check that the Frobenius classes of prime ideals in $S$ realize every conjugacy class $C\subset \Gal(K/k)$ as long as $\mu^1_{K/k}(C)>0$. This completes the proof.
\end{proof}

\begin{Rmk}
However, sometimes the class number $h_K$ is not easy to compute (when $n\gg 0$). In this case, one can use the following estimation \cite[Theorem~6.5]{Len92}:
\begin{align}\label{bound class number}
    h_K\le \frac{d(n-1+\log d)^{n-1}}{(n-1)!},
\end{align}
where $d=(2/\pi)^s\sqrt{|\Delta_K|}$ and $s$ the number of complex embeddings of $K$.
\end{Rmk}

\subsection{An example of testing the non-splitting of a Hilbert exact sequence}\label{Sect: test_non_split}
In this subsection, we apply the method of Theorem~\ref{Thm: effective_test_nonsplit} to a concrete example, and verify the nonsplitting of the associated Hilbert exact sequence. We do not claim that our method is better than the theoretic approach, but we do hope that the method will prove to be useful in cases when the Galois group $\Gal(H_K/k)$, or the Hilbert class field $H_K$, is hard to compute.

\begin{eg}\label{eg: non_split_eg}
Let $K=\Q(\sqrt{-3}, \sqrt{13})$. It is clear that $K$ is Galois over $k=\Q$ with Galois group $\Gal(K/\Q)\simeq \Z/2\times \Z/2$. One can check that $|\Delta_K|=1521$, and $h_K=2$. One of its three quadratic subfields is $L=\Q(\sqrt{-3\times 13})$. Assume that $\sigma$ generates the Galois subgroup $\Gal(K/L)$ and take $C=\{\sigma\}$ to be the corresponding conjugacy class. Using Galois theory, the sub exact sequence 
$$
1\to Cl_K\to E_{\sigma}\to \langle \sigma\rangle\to 1
$$
splits if and only if one can find a prime integer $p$ satisfying 
\begin{enumerate}
    \item unramified in $K$;
    \item totally split in $L$;
    \item not totally split in $K$;
    \item factors principally in $K$. 
\end{enumerate}
Moreover, applying Theorem~\ref{Thm: effective_bound} (or equivalently, Theorem~\ref{Thm: effective_test_nonsplit}), if such a $p$ exists, it can be found under 
$$
 B_K= (4\times 2\times \log|1521|+2.5\times 4\times 2+5)^2< 6992.
$$ However, by directly checking with the help of computer, it is not hard to verify that there is no such prime integer. Hence $\mu_{K/\Q}^1(\sigma)=0$, and the associated Hilbert exact sequence does not split.  
\end{eg}

\begin{Rmk}
Although Yu \cite[Corollary~3.10]{Yu-splitting-Hilbert-class-fields} claims a generalization of this result for $K/k$ is abelian, it seems this is a counter example to his result. We appreciate Siman Wong for pointing this out.
\end{Rmk}

	\bibliographystyle{amsalpha}
	\bibliography{mybib}

\end{document}